%% file: main.tex
\documentclass[12pt]{article}

\usepackage{amsmath}
\usepackage{amsmath}
\usepackage{verbatim}
\usepackage{amssymb}
\usepackage{latexsym}
\usepackage{latexsym}
\usepackage{mathtools}
\usepackage{amsthm}

\usepackage{color}
\usepackage{mathabx}
\definecolor{green}{rgb}{0,1,0}
\usepackage{cancel}

\usepackage{url}
\usepackage{xcolor}
\definecolor{linkColor}{RGB}{156,78,13}
\usepackage{hyperref}
\hypersetup{colorlinks,allcolors=linkColor}

\hypersetup{
           breaklinks=true,   
           colorlinks=true,   
           pdfusetitle=true,  
        }
\begin{document}
\newfont{\blb}{msbm10 scaled\magstep1} 

\newtheorem{Theorem}{Theorem}[section]
\newtheorem{defi}[Theorem]{Definition}
\newtheorem{Remark}[Theorem]{Remark}
\newtheorem{prop}[Theorem]{Proposition}
\newtheorem{lemm}[Theorem]{Lemma}
\newtheorem{coro}[Theorem]{Corollary}

\pagestyle{myheadings}
\date{}
%
%
\makeatletter
\def\@fnsymbol#1{\ensuremath{\ifcase#1\or
    *\or
    \dagger\or
    \ddagger\or
    \mathsection\or
    {\vdash}\or
    \else\@ctrerr\fi}}
\makeatother
\title{\fontfamily{lmss}\selectfont  On minimal \\
symplectic alternating algebras}
\author{
   Layla Sorkatti~\thanks{Southern Illinois University, USA; University of Khartoum, Al Neelain University, Khartoum, Sudan. 
   layla.sorkatti(at)bath.edu},
\"{O}zlem U\u{g}urlu ~\thanks{Saint Louis University, USA, ozlem.ugurlu(at)slu.edu},
Manisha Varahagiri~\thanks{University of Texas at Arlington, USA, manishav(at)uta.edu}
}
\maketitle
\begin{abstract}
\noindent
The structure of nilpotent symplectic algebras of maximal class has been studied in~\cite{sor-thesis, nsaa1}.
In this paper, we focus on the dual subclass of algebras of minimal class. In particular, we show that symplectic alternating algebras of dimension up to $16$ that are minimal, in the sense that they are of rank $2$ with minimum nilpotency class, have a class that confirm a conjecture that has been raised in~\cite{sor-2022}.
\mbox{}\\
\end{abstract}
%
\section{Introduction}
\input{Survay}

\section{Lower bounds for nilpotency classes}
\input{nilpotent-algebras}

\noindent
\section{Key Technical Lemmas}
\input{technical-Lemma}
\section{Main Theorem}
\input{Proof}
\section*{Acknowledgments}
Part of this work has been done while the first author was visiting at Saint Louis University and she is grateful for the great hospitality of the Department of Mathematics and Statistics during the visit, which was supported by the department and AMS-Simons Travel Grants. First and third author also received funding by the Southern Illinois University.\\
\bibliographystyle{plainurl}
\bibliography{biblio}
\end{document}

%% file: Survay.tex
\noindent
Symplectic alternating algebras originate from the study of classifying powerful $2$-Engel groups of class $3$ that have the property that all powerful proper subgroups are nilpotent of class $2$ ~\cite{moravec,gt-2008}
The motivation for this work comes from well-known question raised by Caranti~\cite{unsolved} as follows:
{
\begin{itemize}
 \item[] Does there exist a finite 2-Engel 3-group of class 3 such that \( \mbox{Aut}\ G = \mbox{Aut}_c\ G \cdot \mbox{Inn}\ G \) (where \(\mbox{Aut}_c \ G\) denotes the central automorphisms of G)?
\end{itemize}}

\noindent
The study reveals that there are infinitely many minimal groups of rank $5$ and of any even rank \(\geq 4 \). One of the families considered (of rank $5$ and exponent $27$) has a richer structure that leads to a related algebraic structure that we call \textquoteleft \textquoteleft symplectic alternating algebras\textquoteright\textquoteright. Further, there is a one-to-one correspondence between symplectic alternating algebras over the field of three elements and a certain rich class \( \mathcal{C}\) of powerful $2$-Engel $3$-group of exponent $27$. Namely, the groups form the class \( \mathcal{C}\) consist of all powerful $2$-Engel $3$-groups \(G\) with the following extra properties:
{
\begin{itemize}
    \item[1)] \( G = \langle x, H \rangle \), where \(H  = \{ g \in G \colon g^9=1 \} \) and \( Z(G) = \langle x \rangle \) with \( O(x)=27\),
    \item[2)] \(G\) is of rank \(2r+1\) and has order \(3^{3+4r}\). \\
\end{itemize}}

\noindent
There is another important subclass of \( \mathcal{C}\) that was introduced by Sorkatti and Traustason in~\cite{sor-thesis, nsaa2}. This subclass consists of all groups in \( \mathcal{C}\) that has an additional group-theoretical property that we call {\it powerfully nilpotent}. 
In fact, Traustason showed in~\cite{gt-2008} that \(G \cong H \) in \( \mathcal{C}\) if and only if the associated symplectic alternating algebras \(L(G) \cong L(H)\). A brief overview of the theory of symplectic alternating algebras can be found in ~\cite{survay}.\\\\

\noindent
Symplectic alternating algebras form a class of non-associative algebras that are also of interest in their own right, with many beautiful properties. A symplectic alternating algebra (SAA) \(L\) over a field is a vector space $L$ associated with a non-degenerate alternating form \( ( \ , \ )\) that is further equipped with a binary alternating product `\( \ \cdot \ \)' with the additional requirement that \[ (x \cdot y, z) = ( y  \cdot z, x), \]
for all \( x, y, x \in L\). This condition can be equivalently expressed by saying that $( u \cdot x , v) = (u, v \cdot x)$ for all $u, v, x \in L$ or in other words, that multiplication from the right is self-adjoint with respect to the alternating form.\\ \\
We can always choose a basis $x_1, y_1, \ldots, x_n, y_n$ of $L$ such that $(x_i, x_j)= (y_i, y_j)=0$ and $(x_i, y_j) = \delta_{ij}$ for $1 \leq i \leq j \leq k \leq n$. A basis with this property is called a \emph{standard basis}.
A \emph{presentation} for $L$ on any basis 
$u_1, \ldots, u_{2n}$ is determined by 
\[ (u_i u_j, u_k)= \gamma_{ijk}, \quad 1 \leq i <  j  <  k \leq 2n,\]
where by convention one list only those triples whose values are not zero.

\noindent
Alternatively we can describe $L$ as follows: if we take the two isotropic subspaces $\mathbb{F} x_1+ \cdots + \mathbb{F} x_n$ and $\mathbb{F} y_1+ \cdots + \mathbb{F} y_n$ with respect to a given standard basis, then it suffices to write down only the products $x_ix_j, y_iy_j,$ $1 \leq i < j \leq n$. The reason for this is that having determined these products, we have determined all the triples $ (u_i u_j, u_k)$ where $ 1 \leq i < j < k \leq 2n$, since two of them are either some $x_i, x_j$ or some $y_i, y_j$ in which case the triple is determined from $x_ix_j$ or $y_iy_j$. Since $(x_i x_j, x_k) = (x_j x_k, x_i) = (x_k x_i, x_j)$ and $(y_i y_j, y_k) = (y_j y_k, y_i) = (y_k y_i, y_j)$, this puts additional conditions on the products $x_ix_j$ and $y_iy_j$. \\

\noindent
Classifying symplectic alternating algebras of dimension $2n$ over a field $\mathbb{F}$ is equivalent to finding all the $\mbox{Sp(V)}$-orbits of $\bigwedge^3 V$ under the natural action, where $V$ is a symplectic vector space of dimension $2n$ with a non-degenerate alternating form. In particular, it has been shown in a joint work with Trautason in~\cite{nsaa1} that the number of symplectic alternating algebras over a finite field is $|\mathbb{F}|^{\frac{4}{3}n^2 +O(n^2)}$. Because of this sheer growth, a general classification of symplectic alternating algebras seems impossible. It is clear that the only symplectic alternating algebra of dimension $2$ is abelian, and when the dimension is $4$ we obtain a non-trivial symplectic alternating algebra apart from the abelian one. Traustason in~\cite{saa} classified the symplectic alternating algebras over $\mbox{GF}(3)$ of dimension up to $6$. There are $31$ such algebras of dimension $6$, of which $15$ are simple. Further, it has been shown in~\cite{saa} that a well-known dichotomy property for Lie algebras also holds for symplectic alternating algebras; namely a symplectic alternating algebra is either semi-simple or has a non-trivial abelian ideal.\\

\noindent
Many of the notions that will be used are analogous to the corresponding notions for related structures. The definition of a nilpotent symplectic alternating algebra causes no problem either and it is just as we expect it to be.\\\\

\noindent
{\bf Definition}. A symplectic alternating algebra $L$ is {\it nilpotent} if there 
exists an ascending
chain of ideals $I_{0},\ldots, I_{n}$ such that 
    $$\{0\}=I_{0}\leq I_{1}\leq \cdots \leq I_{n}=L$$
and $I_{s}L\leq I_{s-1}$ for $s=1,\ldots ,n$. 
The smallest possible $n$ is then called the {\it nilpotence class} of $L$. \\ \\

\noindent
More generally, if $I_{0}\leq I_{1}\leq \ldots \leq I_{n}$ is any chain of ideals of $L$ then we say that this chain is central in $L$ if $I_{s}L\leq I_{s-1}$ for $s=1,\ldots ,n$.

\noindent
We define the lower and upper central series in a similar way to related structures like associative algebras and Lie algebras. Thus we define the lower central series recursively by 
\[ L^1=L  \text{ and } L^{i+1}=L^i L, \]
\noindent
and the upper central series by 
\[ Z_0(L)=\{0\} \text{ and } Z_{i+1}(L)=\{ x \in L : xL \in Z_i(L) \}.\]

\noindent
It is readily seen that the terms of the lower and the upper central series are all ideals of $L$. Traustason in~\cite{saa} proved the following dual relation between the upper and the lower central series:
\begin{align}\label{nice-ppty}
    Z_i(L)=(L^{i+1})^\perp
\end{align}

\noindent
We can introduce the notion of rank since all minimal sets of generators have the same number of elements.\\\\
Let L be a nilpotent symplectic alternating algebra. We say
that $\{x_1 , \ldots , x_r \}$ is a minimal set of generators if these generate $L$ (as an
algebra) and no proper subset of it does it.

\noindent
For nilpotent symplectic alternating algebras, the size of the minimal set of generators is $\text{dim}\ L - \text{dim}\ L^2 = $ \text{dim}\  Z(L)~\cite{nsaa1}.\\\\
{\bf Definition}. This number is defined to be the {\it rank} of $L$ and is denoted $r(L)$.\\\\
Thus, for any a nilpotent symplectic alternating algebra $L$ we have
$$r(L) = \text{dim}\ Z(L).$$
From the duality (\ref{nice-ppty}) and the general fact that for any nilpotent alternating algebra $L$ of dimension greater than one that $r(L)$ can't be equal to one.
It follows that any non-trivial nilpotent symplectic alternating algebra $L$ has the property that the dimension of the center is at least $2$. \\\\

\noindent
The algebras that are of maximal class turn out to have a rigid structure, and for every algebra of dimension $2n\geq 8$, this bound is attained. Further, there is a unique ideal of dimension $k$ for any $0 \leq k \leq 2n$ apart from $k =1, k =n$ and $k =2n - 1$. In particular,
\[  L^k = {Z_{k-1}(L)}^\perp  =  Z_{2n-k-2}(L), \]
for $0 \leq k \leq 2n-3$. More interestingly, It turns out that there are as well always characteristic ideals of dimension $1, n$ and $2n - 1$ when $2n \geq 10$ ( see~\cite{sor-thesis}, Theorem 3.27). Explicitly, let $L$ be a nilpotent symplectic alternating algebra of dimension $2n\geq 8$ with an ascending chain of isotropic ideals

    $$\{0\}=I_{0}<I_{1}<\cdots <I_{n},$$ 
where $\mbox{dim\,}I_{j}=j$ for $j=1,\ldots ,n$. In ~\cite{sor-thesis} we proved that for the algebra $L$ that is of maximal
class. We have that

    $$I_{2}=Z_{1}(L),\ I_{3}=Z_{2}(L),\ \ldots ,I_{n-1}=Z_{n-2}(L),$$
   $$I_{n-1}^{\perp}=Z_{n-1}(L),\ I_{n-2}^{\perp}=Z_{n}(L),\ \cdots, I_{2}^{\perp}=Z_{2n-4}(L).$$

\noindent
Furthermore $Z_{0}(L), Z_{1}(L), \ldots ,Z_{2n-3}(L)$ are the unique ideals
of $L$ of dimensions $0, 2, 3, \ldots , n-1, n+1, n+2, \ldots , 2n-2, 2n$. 
In particular when $2n \geq 10$ we can choose our chain of ideals above such that they are all characteristic. It turns out that $I_0, I_2, I_3, \ldots, I_{n-1}, I_{n-1}^\perp, I_{n-2}^\perp,$ $ \ldots, I_0^\perp$ are unique and equal to both the terms of the lower and upper central series (see ~\cite{sor-thesis}, Theorem $3.2$). The nilpotent algebras of maximal class can be identified easily from their presentations. In fact, if ${\mathcal P}$ is any 
presentation of $L$ with respect to a standard basis $\{ x_1, y_1, \ldots, x_n, y_n \}$, and $2n \geq 8$, then $L$ is of maximal class if and only if $x_i y_{i+1} \neq 0$ for all $i = 2, \ldots, n-2$, and $x_1y_2, y_1y_2$ are linearly independent (see ~\cite{sor-thesis}, Theorem $3.4$).\\\\

\noindent
Moreover, we obtain some information about the growth of nilpotent symplectic alternating algebras: It is too large and thus a general classification for this subclass seems to be impossible. In fact Sorkatti and Traustason showed that if $\mathcal{N}_n (\mathbb{F})$ is the number of nilpotent symplectic alternating algebras of dimension $2n$ over a finite field $\mathbb{F}$ then $\mathcal{N}_n(\mathbb{F}) = |\mathbb{F}|^{n^3 /3+O(n^2 )}$~\cite{nsaa1}. 
Clearly we have that $\mathcal{N}_0(\mathbb{F}) = \mathcal{N}_1(\mathbb{F}) = \mathcal{N}_2(\mathbb{F})=1$, $\mathcal{N}_3(\mathbb{F})=2$, and $\mathcal{N}_4(\mathbb{F})=3$. For the next higher dimensions the classification is already challenging. When $\mathbb{F} = GF(3)$, we have $\mathcal{N}_5(\mathbb{F})=25$ while $\mathcal{N}_5(\mathbb{F})=22$ over a field $\mathbb{F}$ that is algebraically closed. Thus, we have $25$ powerfully nilpotent groups of rank $11$ in class \( \mathcal{C}\)~\cite{sor-thesis, nsaa2}.
Surprisingly, the growth of the subclass of nilpotent symplectic alternating algebras of dimension $2n$ over a finite field $\mathbb{F}$ that is furthermore of maximal class is again $|\mathbb{F}|^{n^3 /3+O(n^2 )}$.\\\\

\noindent
Now, let $\mathcal{N}^2_n(\mathbb{F})$ be the set of nilpotent symplectic alternating algebras of dimension $2n$ over a finite field $\mathbb{F}$ that have rank 2. We can define the two functions $F (n) = \text{max}\, \{ \text{class}(L) \colon L \in
\mathcal{N}^2_n(\mathbb{F})\}$ and $f(n) = \text{min} \, \{ \text{class}(L) \colon L \in \mathcal{N}^2_n(\mathbb{F})\}$. 
In ~\cite{nsaa1}, it was proved that $F(n)=2n-3$ for any nilpotent symplectic alternating algebras of dimension $2n \geq 8$.
It comes as a natural question to consider the funtion $f(n)$ and investigate the lower bound of the nilpotency class for the finite dimensional algebras of rank $2$ over a field $\mathbb{F}$. 

\noindent
This motivates the following questions \\\\
{\bf Question A}~\cite{sor-2022}. What can we say about the structure of symplectic alternating algebras that are minimal in the sense that they are of ranks $2$ with minimal nilpotency class? \\\\

\noindent
The growth of nilpotent symplectic alternating algebras is the same as the growth of nilpotent symplectic alternating algebras that are of maximal class. Because of the sheer growth, one might therefore expect that all the examples of minimal class would disappear if we add the condition that rank $L$ is two. Surprisingly, this is not the case in general.\\\\
\noindent
{\bf Question B}~\cite{sor-2022}. 
Does the class of nilpotent symplectic alternating algebras of rank $2$ and dimension $2n \geq 8$ over any field have at least $2m+1$ or $2m+2$, where $m$ is a certain non-negative integer?\\\\ 

\noindent
Our main result is the following theorem.
\begin{Theorem} Let $\mathbb{F}$ be a field. Let $L$ be any minimal symplectic alternating algebra that is of dimension $2n \leq  16$ over $\mathbb{F}$. Then $L$ has nilpotency class $\leq 7$ that is same as the class given in Question B.
\end{Theorem}
\noindent
The outline of the paper is as follows. 
Section $2$ is devoted to some necessary and crucial results from the theory of symplectic algebras. 
The main technical result of the paper concerns the algebras of nilpotency class $5$ and $6$, and is proved in Section $3$.
Finally, in Section $4$, we prove the main theorem. 
Our result thus confirms the conjecture given in the above Question B for algebras of dimension up to $16$ over any field.\\

%% file: nilpotent-algebras.tex
Before embarking into the necessary material needed as a preparation to prove Theorem 1.1,
we first see the sufficient condition to have a symplectic alternating algebra~\cite{saa}.
Suppose that $V$ is a symplectic vector space with a non-degenerate alternating form $(\ ,\ )$. Let $x_1 , \ldots , x_{2n}$ be a standard basis of $V$, that is $(x_i, x_j)= (y_i, y_j)=0$ and $(x_i, y_j) = \delta_{ij}$ for $1 \leq i \leq j \leq k \leq n$. 
\begin{lemm} Let $n \geq 2$ and for each $(i,j,k)$, $ 1 \leq i < j < k \leq 2n$, choose a number $\alpha(i,j,k)$ in the field $F$. There is a unique SAA of dimension $2n$ over $F$ satisfying $$ (x_i x_j, x_k) = \alpha(i,j,k)$$ for $1 \leq i < j <k \leq 2n.$
\end{lemm}
\noindent
Next, we state one of the main crucial results that provide a great deal about the structure of the nilpotent subclass.
\begin{prop}
Let $L$ be a nilpotent SAA of dimension $2n \geq 2$. There exists an ascending chain of isotropic ideals
\[\{0\}=I_0 < I_1 < \cdots < I_{n-1} < I_n \]
such that $\mbox{dim\,} I_r = r$ for $r = 0, \ldots, n$. 
Furthermore, for $2n \geq 6$, $I_{n-1}^\perp$ is abelian and the ascending chain 
\[\{0\} < I_2 < I_3 < \cdots < I_{n-1} < I_{n-1}^\perp < I_{n-2}^\perp  < \cdots< I_2^\perp < L \]
is a central chain. 
In particular, $L$ is nilpotent of class at most $2n-3$.
\end{prop}
\noindent
Next, we apply the previous theorem to the standard basis $x_1, y_1, x_2, y_2,$
$\ldots, x_n, y_n$ where
\[\mbox{}\mbox{}\mbox{}\mbox{}I_1={\mathbb F}x_n,\  I_2={\mathbb F}x_n + {\mathbb F} x_{n-1},\  \ldots,\  I_n={\mathbb F}x_n+\cdots+{\mathbb F}x_1,\]
\[{I^\perp_{n-1}}=I_n+{\mathbb F}y_1,\  {I^\perp_{n-2}}=I_n+{\mathbb F}y_1+{\mathbb F}y_2,\  \ldots,\  {I^\perp_0}=L=I_n+{\mathbb F}y_1+\cdots+{\mathbb F}y_n.\]
Now, let $u, v, w$ be three of the basis elements. Since $I_n$ is abelian we have that $(uv, w) = 0$ whenever two of these three elements are from $\{ x_1, \ldots, x_n\}$. 
The fact that
\[ \{0\} < I_1 < \cdots < I_n \]
is central also implies that $(x_iy_j, y_k) =0 $ if $i \geq k$.
So we only need to consider the possible non-zero triples $(x_iy_j,y_k),\  (y_iy_j,y_k)$ for $1\leq i < j < k \leq n $. 
For each triple $(i, j, k)$ with $1 \leq i < j < k \leq n$, let $\alpha(i, j, k)$ and $\beta(i, j , k)$ be some elements in the field $\mathbb{F}$. 
We refer to the data  
\[
 \quad (x_iy_j,y_k)=\alpha(i, j, k),\quad (y_iy_j,y_k)=\beta(i, j, k), \quad 1 \leq i < j < k \leq n 
\]
as a \textit{nilpotent presentation}. We have just seen that every nilpotent SAA has a presentation of this type. 
Conversely, given any nilpotent presentation, let
 \[I_r={\mathbb F} x_n+{\mathbb F} x_{n-1}+\cdots+{\mathbb F} x_{n+1-r}\]
and we get an ascending central chain of isotropic ideals $\{0\}=I_0 < I_1 < \cdots < I_n$ such that $\mbox{dim\,} I_j = j$ for $j = 1, \ldots, n$. 
It follows that we then get a central chain 
\[\{0\}=I_0 < I_1 < \cdots  < I_n  < I^\perp_{n-1} < I^\perp_{n-2} < \cdots < I^\perp_0=L \]
and thus $L$ is nilpotent. 
Thus, every nilpotent presentation describes a nilpotent SAA.\\\\
\noindent
We define the lower and upper central series in a similar way to related structures like associative algebras and Lie algebras. Thus, we define the lower central series recursively by 
\[ L^1=L  \text{ and } L^{i+1}=L^i L, \]
\noindent
and the upper central series by 
\[ Z_0(L)=\{0\} \text{ and } Z_{i+1}(L)=\{ x \in L : xL \in Z_i(L) \}.\]
It is readily seen that the terms of the lower and the upper central series are all ideals of $L$.
In fact, the dimensions of the upper and lower central series can be described by the following useful inequality (see ~\cite{sor-2022}).  
\begin{lemm}
\noindent
Let $L$ be a nilpotent symplectic alternating algebra. Then
\begin{align*}
\text{dim\,} Z_i(L)  - \mbox{dim\,} Z_{i-1}(L) & \leq  \frac{1}{2}(\mbox{dim\,} Z_{i-1}(L)-\mbox{dim\,} Z_{i-2}(L))(\mbox{dim\,} Z_{i-1}(L) \\
& + \mbox{dim\,} Z_{i-2}(L)-1).
\end{align*}
\end{lemm}
\begin{Remark}. Equivalently, from the previous inequality and since $$\mbox{dim\,} L^{i} \, -\, \mbox{dim\,} L^{i+1} = \mbox{dim\,} ({L^{i+1}})^\perp \, - \,  \mbox{dim\,} ({L^i})^\perp = \mbox{dim\,} Z_i(L) \, -\,  \mbox{dim\,} Z_{i-1}(L),$$
it follows that $\mbox{dim\,} L^{i} \, -\, \mbox{dim\,} L^{i+1} $ has the same upper bound.\\
Thus for a nilpotent symplectic alternating algebra $L$ where $\mbox{dim\,} L > 2$ and $\mbox{dim\,} Z(L)=2$ we have:
\begin{align*}
2 < \mbox{dim\,} Z_2(L) \leq 2+1=3  &\Rightarrow \mbox{dim\,} Z_2(L)=3 \Leftrightarrow \mbox{dim\,} L^3 = dim L^2 -1.\\
3 < \mbox{dim\,} Z_3(L) \leq 3+2=5  &\Rightarrow \mbox{dim\,} Z_3(L) \in \{4,5\} \\
&\Leftrightarrow \mbox{dim\,} L^4 \in \{\mbox{dim\,} L^3-1,\mbox{dim\,} L^3-2\}.\\
\end{align*}
\end{Remark}
\noindent
Now we introduce some notations that will be needed in the sequel. 
Let $\Omega \colon {\mathbb N} \to {\mathbb N}$ be a function defined recursively by
$\Omega(0)=0,\, \Omega(1)=2,$ and 
$$\Omega (n+1)= 2 + {\Omega (n) \choose 2}.$$
\noindent
Let $\Omega(m) < n \leq \Omega(m+1)$. We consider two cases:
$\Omega(m) < n \leq \frac{\Omega(m)+\Omega(m+1)}{2}$ and $\frac{\Omega(m)+\Omega(m+1)}{2} < n \leq \Omega(m+1)$.\\\\
First we deal with the case of algebras where $\Omega(m) < n \leq \frac{\Omega(m)+\Omega(m+1)}{2}$.
Notice that as $\Omega(m) < n \leq \frac{\Omega(m)+\Omega(m+1)}{2}$, $2n - \Omega(m) \leq \Omega(m+1)$ and thus \begin{equation*} 2n-2-\Omega(m) \leq {\Omega(m) \choose 2}. \end{equation*}
\noindent
Let $V$ be a symplectic vector space of dimension $2n \geq 8$ whose alternating form is non-degenerate and consider a standard basis $x_1, y_1, \ldots, x_n, y_n$.
We consider the following sets:
\begin{align*}
\mathcal{U}:=\, & \{ x_1, y_1, \ldots, x_{n-\Omega(m)}, y_{n-\Omega(m)} \}\\
 V_{\Omega(r)}:=\, & \{x_n, x_{n-1}, \ldots, x_{n-\Omega(r)+1} \}\\
V^{(2)}_{\Omega(r)}:=\, &\{ (y_i,y_j) \colon n-\Omega(r)+1 \leq i < j \leq n \},
\end{align*}
where $1 \leq r \leq m$ and $\mathcal{W} \subseteq V^{(2)}_{\Omega(m)} \setminus V^{(2)}_{\Omega(m-1)}$.
It is not difficult to see that there is a bijection $$\psi_{r} \colon V_{\Omega(r+1)} \setminus
 V_{\Omega(r)} \to V^{(2)}_{\Omega(r)} \setminus V^{(2)}_{\Omega(r-1)}, u \mapsto (v,w).$$
Further, as
$2 (n - \Omega(m)) \leq \Omega(m+1)-\Omega(m) = {\Omega(m) \choose 2} \setminus {\Omega(m-1) \choose 2} = |V^{(2)}_{\Omega(m)} \setminus V^{(2)}_{\Omega(m-1)}|,$
there is an injection 
$$\psi_{m} \colon U \to V^{(2)}_{\Omega(m)}
 \setminus V^{(2)}_{\Omega(m-1)}, z \mapsto (x,y).$$ 
Notice that if $\Omega(m)-\Omega(m-1) \leq 2$ then $2 \leq |U| \leq 6$ and we choose $\psi_{m}$ such that $\psi_{m}(y_1)$ (if needed) involves all elements in 
$\{ y_{n-\Omega(m)+1}, \ldots, y_{n-\Omega(m-1)} \}$. Then
there is a natural bijection $$\Psi \colon V_{\Omega(m)} \bigcup U \setminus V_{\Omega(1)} \to 
V^{(2)}_{\Omega(m-1)} \bigcup \psi_{m}(u), x \mapsto (v,w).$$\\
Now consider the set
$$T^{(3)} = \{ (u,v,w) \colon u \in V_{\Omega(m)} \bigcup \mathcal{U}
 \setminus V_{\Omega(1)}, \, (v,w) =\Psi (u)\in 
 V^{(2)}_{\Omega(m-1)} \bigcup \mathcal{W}\}.$$
It follows that we can pick our standard basis $x_1, y_1, \ldots, x_n, y_n$ such that each triple in $T^{(3)}$
is a standard, where each triple is either of the form $(y_i y_j, y_k)$ or of the form $(x_i y_j, y_k)$, with the following extra properties:
\begin{itemize}
\item[(1)]\ $u \in \{ x_{n-2}, x_{n-3},  \ldots, x_1 , y_1, \ldots,  y_{n-\Omega(m)}  \}$ 
and $ (v,w) \in V^{(2)}_{\Omega(m)}$.\\
\item[(2)]\ For $ 1 \leq r \leq m-1$ we have 
\begin{itemize}
\item[(i)]\ For each $u \in {V}_{\Omega(r+1)} \setminus 
{V}_{\Omega(r)}$ there is a unique $(v,w)$ such that $(u,v,w) \in
 T^{(3)}$ and $(v,w) \in V^{(2)}_{\Omega(r)} \setminus 
V^{(2)}_{\Omega(r-1)}$, conversely
\item[(ii)]\ For each $(v,w) \in {V^{(2)}}_{\Omega(r)} \setminus 
{V^{(2)}}_{\Omega(r-1)}$ there is a unique $u$ such that
 $(u,v,w) \in T^{(3)}$ and $ u \in V_{\Omega(r+1)} 
\setminus V_{\Omega(r)}$.\\
\end{itemize}
\item[(3)]\ No two triples in $T^{(3)}$ have two common entries.\\
\noindent
\noindent
\item[(4)]\ All of $x_{n-2}, x_{n-3}, \ldots, x_1, y_1, \ldots, y_n$ are involved in some triple. 
\\
\end{itemize}
In particular, all elements in $\{x_{n-2}, \ldots, y_{n}\}$ are involved in a standard triple.
We now have acquired all the necessary ingredients to obtain a symplectic alternating algebra of rank $2$ and dimension $2n \geq 8$ of class $2m+1$ over any field $\mathbb{F}$ as described in the next result. More details can be found in [~\cite{sor-2022}, Theorem 3.6].\\
\begin{prop} Let $\mathbb{F}$ be a field and let $L$ be a symplectic vector space of dimension $2n \geq 8$ with non-degenerate alternating form such that 
$\Omega(m) < n \leq \frac{\Omega(m)+\Omega(m+1)}{2}$. Then there exists a nilpotent symplectic alternating algebra $L$ of rank $2$ and dimension $2n \geq 8$ over $\mathbb{F}$ that can be given by the
presentation
 $${\mathcal P}:\ (u v, w) =1,\ \forall (u,v,w) \in T^{(3)}$$
 where the nilpotent class of $L$ is $2m+1$.\\
\end{prop}
\noindent
Turning to the second case of algebras where $\frac{\Omega(m)+\Omega(m+1)}{2} < n \leq \Omega(m+1)$.
Then $\frac{\Omega(m+1) - \Omega(m)}{2} <   n - \Omega(m) \leq \Omega(m+1) - \Omega(m)$ and as $\Omega(m+1) - \Omega(m) > \Omega(m) - \Omega(m-1)$, we have
\begin{equation*} n - \Omega(m) > \frac{\Omega(m) - \Omega(m-1)}{2}. \end{equation*}\\
We show that there exist an algebra $L$ of rank $2$ and dimension $2n \geq 8$ over any field $\mathbb{F}$ such that the nilpotence class of $L$ is $2m+2$. First let $V$ be a symplectic vector apace with a standard basis $x_1, y_1, \ldots, x_n, y_n$ and consider the sets
\begin{eqnarray*}
 V_{\Omega(r)} &:=& \{x_n, x_{n-1}, \ldots, x_{n-\Omega(r)+1} \},\\
V^{(2)}_{\Omega(r)} &:= & \{ (y_i,y_j) \colon n-\Omega(r)+1 \leq i < j \leq n \},\\
 \mathcal{U} &:= & \{  x_{n-\Omega(m)}, x_{n-\Omega(m) -1}, \ldots, x_1 \}, \\
\mathcal{W} &:= & \{  y_{n-\Omega(m)}, y_{n-\Omega(m) -1}, \ldots, y_1 \}, 
\end{eqnarray*}
where $1 \leq r \leq m$.
As in the previous case for all $ 1 \leq r \leq m-1$, there is a bijection $$\psi_{r} \colon V_{\Omega(r+1)} \setminus
 V_{\Omega(r)} \to V^{(2)}_{\Omega(r)} \setminus V^{(2)}_{\Omega(r-1)}, u \mapsto (v,w).$$
We also have an injection 
$$\psi_{m} \colon \mathcal{U} \to V^{(2)}_{\Omega(m)} \setminus V^{(2)}_{\Omega(m-1)}, u \to (v,w).$$
As $n - \Omega(m) > \frac{\Omega(m+1) - \Omega(m)}{2}$, we choose the injection $\psi_{m}$ such that $\mathcal{U}$ gets mapped to a pair that involved all the elements in $$ y_{n - \Omega(m-1)},   \ldots,    y_{n - \Omega(m)+1}.$$ Therefore
there is a natural bijection 
$$ \Psi \colon V_{\Omega(m)} \bigcup \mathcal{U} \setminus V_{\Omega(1)} \to T^{(2)}:= V^{(2)}_{\Omega(m -1)} \bigcup \psi_{m}(\mathcal{U}), u \to (v,w).$$
It follows that each triple in
$$T^{(3)} = \{ (u,v,w) \colon u \in V_{\Omega(m)} \bigcup \mathcal{U}
 \setminus V_{\Omega(1)}, \text{ and } (v,w) = \Psi (u) \in 
V^{(2)}_{\Omega(m -1)} \bigcup \psi_{m}(\mathcal{U}) \}$$
is a standard triples with the following extra properties that
\begin{itemize}
\item[eq-l4-gen]\ $u \in \{ x_{n-2}, x_{n-3},  \ldots, x_1\}$ and $ (v,w) \in V^{(2)}_{\Omega(m)} \cup W.$\\
\item[(2)]\ For $ 1 \leq r \leq m-1$ we have 
\begin{itemize}
\item[(i)]\ For each $u \in {V}_{\Omega(r)} \setminus 
{V}_{\Omega(r-1)}$ there is a unique $(v,w)$ such that $(u,v,w) \in
 T^{(3)}$ and $(v,w) \in V^{(2)}_{\Omega(r-1)} \setminus 
V^{(2)}_{\Omega(r)}$, conversely
\item[(ii)]\ For each $(v,w) \in V^{(2)}_{\Omega(r-1)} \setminus 
V^{(2)}_{\Omega(r)}$ there is a unique $u$ such that
 $(u,v,w) \in T^{(3)}$ and $ u \in V_{\Omega(r)} 
\setminus V_{\Omega(r-1)}$.
\item[(iii)]\ For each $x_i \in U$, there is a unique $(v,w)$ such that $(u,v,w) \in T^{(3)}$ and $(v,w) \in V^{(2)}_{\Omega(m)} \setminus V^{(2)}_{\Omega(m-1)}$.\\
\end{itemize}
\item[(3)]\ No two triples in $T^{(3)}$ have two common entries.\\
\item[(4)]\ All of $x_{n-2}, x_{n-3}, \ldots, x_1, y_1, \ldots, y_n$ occurs as an entry of some triple in $T^{(3)}$.\\
\end{itemize}
\noindent
Therefore, all elements apart from the central elements are engaged in some standard triples. We are now ready to describe a subclass of the corresponding algebras for this case, see Theorem 3.9 in~\cite{sor-2022}, for more details.

\noindent
\begin{prop} Let $L$ be a symplectic vector space of dimension $2n \geq 8$ with non-degenerate alternating form such that 
$\frac{\Omega(m)+\Omega(m+1)}{2} < n \leq \Omega(m+1)$. Then there exists a nilpotent symplectic alternating algebra $L$ of rank $2$ and dimension $2n \geq 8$ over a field $\mathbb{F}$ that is of class $2m+2$. This algebra can be given by the
presentation
$${\mathcal P}:\ (u v, w) =1,\ \text{ for all } (u,v,w) \in T^{(3)}$$
\end{prop}

%% file: technical-Lemma.tex
\noindent
As we pointed out in the Introduction, we study symplectic alternating algebras of dimension $2n$ over any field $\mathbb{F}$ that are minimal, in the sense that they are of rank $2$ with minimum nilpotency class.\\\\

\noindent
In ~\cite{nsaa1} nilpotent symplectic algebras of dimension $8$ have been classified, whereas~\cite{nsaa2, nsaa3} reveals the classification of nilpotent algebras of dimensions up to $10$. 
From the classification, we know that there is only one algebra of dimension $8$ that is of both maximal and minimal nilpotency class and seven algebras of dimension $10$ with an isotropic center of dimension $2$. Then in ~\cite{sor-2022} we identified nilpotent algebras of $2n \leq 10$ that are minimal as follows.

\noindent
\begin{lemm} Let $L$ be the minimal nilpotent algebra of dimension $2n \leq 10$. 
Then $L$ has one of the following presentations 
$${\mathcal P}_{8}^{(2,1)}(r):\\ \quad (x_2 y_3, y_4)=r,\ (x_1 y_2, y_3)=1,\ (y_1y_2,y_4)=1.$$
$${\mathcal P}_{10}^{(2,1)}:\ \ (x_{3}y_{4},y_{5})=1,\ (x_{2}y_{3},y_{5})=1,\ 
      (x_{1}y_{3},y_{4})=1,\ (y_{1}y_{2},y_{5})=1.$$
 $${\mathcal P}_{10}^{(2,2)}(r):\ \ (x_{3}y_{4},y_{5})=r,\ 
      (x_{2}y_{3},y_{5})=1,\ (x_{1}y_{3}y_{4})=1,\ (y_{1}y_{2},y_{3})=1,$$
where $r\not =0$. Furthermore 
${\mathcal P}_{8}^{(2,1)}(r) \cong {\mathcal P}_{8}^{(2,1)}(s)$ if and only if $s/r \in (\mathbb{F}^* )^3$ and conversely any algebra with such a presentation has the properties stated.
We also have ${\mathcal P}_{10}^{(2,2)}(r) \cong {\mathcal P}_{10}^{(2,2)}(s)$ if and only if $s/r\in ({\mathbb F}^{*})^{4}$ and, conversely, any algebra with such a presentation has the properties stated.
\end{lemm}
\begin{proof} See Propositions 2.2 and 2.3 in~\cite{sor-2022}.
\end{proof}

\noindent
Turning to the subclass of algebras of dimension $2n \geq 12$, we next see that it must have nilpotency class greater than $6$. 
It is worth noting that we have minimal algebras of dimension $8$ and nilpotency class $5$, as well as minimal algebras of dimension $10$ and class $6$.

\noindent
\begin{lemm}
Let $L$ be a nilpotent algebra of dimension $2n \geq 12$ over any field with the property that $\mbox{dim\,} \ Z(L) =2$. Then $L$ must have class at least $7$.
\end{lemm}
\begin{proof}
Notice that $Z(L)$ must be isotropic as otherwise we would have a $2$-dimensional symplectic subalgebra $I$ within $Z(L)$. This would imply that we have a direct sum $I \oplus I^\perp$ of two symplectic alternating algebras. As $I^\perp$ has non-trivial centre this would contradict the assumption that $Z(L)$ is $2$-dimensional. \\\\

\noindent
Now let $L=\mathbb{F}y_n+\mathbb{F}y_{n-1}+ L^2$. Then $$L^2=\mathbb{F}y_n y_{n-1}+L^3$$
It follows that $L^2=Z(L)^\perp$ is of dimension $2n-2$ and that $L^3$ is of dimension $2n-3$. Thus we can pick our standard basis such that

\begin{align*}
Z(L)&=\mathbb{F}x_n+\mathbb{F}x_{n-1},\\
L^2 =Z(L)^\perp &= \mathbb{F}x_n+  \mathbb{F}x_{n-1}+\cdots + \mathbb{F}y_{n-3}+\mathbb{F}y_{n-2},\\
L^3 =Z_2(L)^\perp&= \mathbb{F}x_n+ \mathbb{F}x_{n-1}+\cdots + \mathbb{F}y_{n-3}.
\end{align*}
Thus $L^3=\mathbb{F}y_n y_{n-1}y_n+\mathbb{F}y_n y_{n-1}y_{n-1}+L^4$ and hence
\begin{align}\label{eq-l4-gen}
\mbox{dim\,}\, L^4 \in \{2n-4,2n-5\}. 
 \end{align}

\noindent
In particular, $\mbox{dim\,} \, L^4 \geq 7$ for algebras of dimension $2n \geq 12$.
Now, let $k$ be the nilpotence class of $L$.
As $\mbox{dim\,}\ L^k \neq 1$ and $L^k \leq Z(L)$, we must have that
\begin{align}\label{eq-lk}
 L^k = Z(L).
 \end{align}
Moreover, we know that $\mbox{dim\,}\ L^s \neq 2$ for $1 \leq s \leq 4$ and thus 
\begin{align}
5 \leq k \leq 2n-3. 
\end{align}

\noindent
Notice that we can not have $k=5$ as otherwise $\mbox{dim\,} \ Z_2(L) - \mbox{dim\,} \ Z(L) = \mbox{dim\,} \ L^2 - \mbox{dim\,} \ L^3=1 $ and we get a contradiction that $L^4=Z_2(L)$ is of dimension $3$.\\\\

\noindent
It remains to see that the nilpotency class can not be $6$. To see this, we argue by contradiction and assume that $k=6$. 
Then from (\ref{eq-lk}) the following holds 
\begin{align*}
L^6 = Z(L) &\Longleftrightarrow Z_5 (L)=L^2,\\
L^5=Z_2(L) &\Longleftrightarrow Z_4(L)=L^3.
\end{align*}
We also have that $(L^4, L^4) = (L^7, L) =0$ and hence $L^4$ would be an isotropic (and thus abelian) ideal. Then
\begin{align}\label{l4-eq}
\mbox{dim\,} \, L^4 \leq n. 
\end{align}
It follows from equation (\ref{eq-l4-gen}) and (\ref{l4-eq}) that the algebras with the property that $\mbox{dim\,} \ Z(L) =2$ and class $6$ must be of dimension $10$.
This concludes the proof.
\end{proof}

%% file: Proof.tex
Our aim is to obtain the class for minimal algebras of dimension up to $16$. We show that this nilpotency class agrees with the lower bound given in Question B, that depends only on the dimension of the algebra, as follows:\\\\
Let $\Omega \colon {\mathbb N} \to {\mathbb N}$ be a function defined recursively by
$\Omega(0)=0$, $ \Omega (m+1) = 2 + {\Omega (m) \choose 2}$ and let $\Omega(m) < n \leq \Omega(m+1)$.
Then the class of nilpotent symplectic alternating algebras $L$ of rank $2$ and dimension $2n \geq 8$ over a field $\mathbb{F}$ is 
\begin{equation}\label{eq-main}
    Class\ (L)=
    \begin{cases}
      2m+1, & \text{if}\ \ \Omega(m) < n \leq \frac{\Omega(m)+\Omega(m+1)}{2}. \\
      2m+2, & \text{if}\ \ \frac{\Omega(m)+\Omega(m+1)}{2} < n \leq \Omega(m+1).
    \end{cases}
  \end{equation}
%
We thus prove the above Equation for algebras of dimension up to $16$.\\
\begin{Theorem} Let $L$ be any minimal symplectic alternating algebra that is of dimension $2n \leq 16$ over any field. Then $L$ has nitpotency class satisfying the conditions in Equation (\ref{eq-main}).\\
\end{Theorem}
\begin{proof} As the center of the algebra $L$ is of dimension $2$, $L$ must be of dimension $2n \geq 8$.
From Lemma $3.1$, we see that minimal algebras of dimension $8$ have class $2 \cdot 2 + 1 =5$ and the minimal algebras of dimension $10$ have class $2 \cdot 2 + 2 =6$. This takes care of 
all algebras of dimension $2n \leq 10$. It remains to consider algebras of dimension greater than or equal to $12$. We show that any algebra $M$ of dimension $2n \leq 16$ and rank $2$, has class $M \geq 2m+1=7$.
From Lemma $3.2$, there is no such algebra with dimension $Z(L) =2$ and nilpotency class less than or equal to $6$.\\\\
%
Turning to the algebra of dimension $12$: inspection shows that the following algebra
\begin{align*}
{\mathcal P}_{12}^{(2,1)}:\ \  (x_{4}y_{5},y_{6}) &=1,\ (x_{3}y_{4},y_{6})=1,\\
(x_{2}y_{4},y_{5}) &=1,\ (x_{1}y_{2},y_{4})  =1,\ (y_{1}y_{2},y_{3})=1,
\end{align*}
is of dimension $12$, $\mbox{dim\,}\ Z(L) =2$ and class $7=2 \cdot 3 +1$. \\\\
Next, consider the following algebra
\begin{align*}
{\mathcal P}_{14}^{(2,1)}:\ \  (x_{5}y_{6},y_{7}) &=1,\ (x_{4}y_{5},y_{6})=1,\   (x_{3}y_{5},y_{7})=1,\\
 (x_{2}y_{3},y_{5}) &=1,\ (x_{1}y_{3},y_{6})=1,\ (y_{1}y_{4},y_{5})=1, \ (y_{2}y_{4},y_{6})=1.
\end{align*}
Inspection shows that it is a nilpotent algebra of dimension $14$ and class $7$ as required.
Finally, inspection shows that the following presentation 
\begin{align*}
{\mathcal P}_{16}^{(2,1)}:\ \  (x_{6}y_{7},y_{8}) &=1,\ (x_{5}y_{6},y_{8})=1,\  \\ 
(x_{4}y_{6},y_{7}) & =1, \ (x_{3}y_{5},y_{8})=1,\ (x_{2}y_{5},y_{7})=1,\ \\
(x_{1}y_{5},y_{6}) & =1,\ (y_{1}y_{4},y_{8})=1, \ (y_{2}y_{4},y_{7})=1 , \ (y_{3}y_{4},y_{6})=1,
\end{align*}
\noindent
has nilpotency class $7$, which settle the case for algebras of dimension $16$.
Therefore, (\ref{eq-main}) holds and the result follows. 
\end{proof}